\documentclass[10pt,reqno]{amsart}
\usepackage{amsmath, amsthm}
\usepackage{amssymb,latexsym}
\usepackage{enumerate}

\providecommand{\R}{\mathbb{R}}
\providecommand{\Z}{\mathbb{Z}}
\providecommand{\dist}[1]{\ensuremath{\mathrm{dist}(#1)}}
\providecommand{\abs}[1]{\left|#1\right|}
\providecommand{\norm}[1]{\lVert#1\rVert}
\providecommand{\avg}[2][]{\ensuremath{\mathrm{avg}_{#1}(#2)}}
\providecommand{\grad}{\nabla}
\providecommand{\gs}{\gtrsim} \providecommand{\ls}{\lesssim}
\providecommand{\eqs}{\approx}
\providecommand{\BMO}{\text{BMO}}
\providecommand{\VMO}{\text{VMO}}

\renewcommand{\epsilon}{\varepsilon}
\renewcommand{\rho}{\varrho}

\newtheorem{theorem}{Theorem}[section]
\newtheorem{lemma}[theorem]{Lemma}
\newtheorem{corollary}[theorem]{Corollary}

\theoremstyle{definition}
\newtheorem{definition}{Definition}[section]

\theoremstyle{remark}
\newtheorem{remark}{Remark}

\title[The $L^p$ Dirichlet problem for non-divergence form operators]{The $L^p$ Dirichlet problem for second-order, \\ non-divergence form operators:\\ solvability and perturbation results}

\date{27 January 2011}
\author{Martin Dindo\v{s}}
\address{The University of Edinburgh and Maxwell Institute of Mathematical
Sciences, JCMB, The King's Buildings, Mayfield Road, Edinburgh EH9
3JZ, Scotland} \email{m.dindos@ed.ac.uk}

\author{Treven Wall}
\address{John Hopkins University, Applied Physics Laboratory and Department of Applied Mathematics and Statistics, 11100 Johns Hopkins Road, Laurel, MD 20723, USA} \email{treven.wall@jhuapl.edu}

\keywords{second order non-divergence form elliptic operators,
perturbation theorem, $L^p$ solvability, Dirichlet problem}

\thanks{2010 {\it Mathematics Subject Classification} 35J25, 42B37}
\thanks{First author supported by EPRC grant EP/F014589/1-253000}

\begin{document}

\begin{abstract}
We establish Dahlberg's perturbation theorem for non-divergence
form operators ${\mathcal L}=A\grad^2$. If ${\mathcal L}_0$ and
${\mathcal L}_1$ are two operators on a Lipschitz domain such that
the $L^p$ Dirichlet problem for the operator ${\mathcal L}_0$ is
solvable for some $p\in (1,\infty)$ and the coefficients of the
two operators are sufficiently close in the sense of Carleson
measure, then the $L^p$ Dirichlet problem for the operator
${\mathcal L}_1$ is solvable for the same $p$. This is an improvement of the
$A_\infty$ version of this result proved by Rios in \cite{R1}.
As a consequence we also improve a result from \cite{DPP}
for the $L^p$ solvability of non-divergence form operators
(Theorem \ref{T:main}) by substantially weakening the condition
required on the coefficients of the operator. The improved
condition is exactly the same one as is required for divergence
form operators ${\mathcal L}={\text{div }} A\grad$.
\end{abstract}

\maketitle

\section{Introduction}

This paper is a continuation of a long line of work, most recently
advanced in \cite{DPP}, on the solvability of the $L^p$ Dirichlet
problem for elliptic operators with rough coefficients on
Lipschitz domains with small constants. Here, we extend the
results to non-divergence form operators satisfying a certain
oscillatory Carleson condition which was left open in \cite{DPP}.

In particular, the non-divergence form results in \cite{DPP}
require a condition on the gradient of the coefficient matrix,
since their results arise from considering the divergence form
case first. They then change a non-divergence form operator to
divergence form by allowing first order terms.

Throughout this paper, the operators $\mathcal{L}$ which we
consider are second-order, linear, uniformly elliptic and in
non-divergence form. Precisely, $\mathcal{L} = a^{ij}(x)
\partial_{ij}$ (we use here and throughout the paper the usual
summation convention), where $A(x) = \left(a^{ij}(x)\right)_{i,j}$ is a
symmetric matrix with ellipticity constant $0 < \lambda < \infty$
such that for all $x,\xi \in \R^n$,
\begin{equation}\label{E:ellipticity}
 \lambda \abs{\xi}^2 \le \xi^t A(x) \xi \le \lambda^{-1} \abs{\xi}^2.
\end{equation}
We assume throughout that $n \ge 2$.

The problem under consideration is the Dirichlet problem
\begin{align}\label{E:pde}
    \mathcal{L}u &= 0 \quad \text{ in $D$} \\
    u &= g \quad \text{ on $\partial D$}, \notag
\end{align}
where $D \subset \R^n$ is a bounded Lipschitz domain.

It is a fairly difficult task to even define the notion of a
solution to the equation \eqref{E:pde}. Recall that in the
divergence form case once can use Peron's method to construct a
solution for coefficients that are merely bounded and measurable.
This is not the case in our situation. For this reason we
postulate what we mean by solving the {\it continuous Dirichlet
problem}, denoted ${\mathcal C}{\mathcal D}$, following \cite{R1} and
\cite{R2}:

\begin{definition}\label{D:CD} ({\rm Continuous Dirichlet problem, ${\mathcal C}{\mathcal
D}$}) Given an operator $\mathcal L$ we say that the continuous
Dirichlet problem is uniquely solvable in $D$ (and we say
${\mathcal C}{\mathcal D}$ holds for $\mathcal L$) if for every
continuous function $g$ on $\partial D$ there exists a unique
solution $u$ of \eqref{E:pde} such that $u\in C(\overline{D})\cap
W^{2,p}_{loc}(D)$ for some $1\le p\le \infty$.
\end{definition}

From the results in \cite{CFL}, if the coefficients $a^{ij}$ are
in VMO (see section \ref{S:basic_definitions} for the
definition) and $g \in C(\partial D)$,
then \eqref{E:pde} has a unique solution $u_g \in C(\overline{D})
\cap W^{2,p}_{loc}(D)$ for all $p$, $1 < p < \infty$. By
approximation, one can extend this result to allow for
coefficients in $\BMO_{\rho_0}$ (also defined in section \ref{S:basic_definitions}) in a restricted range
of $p$: $1 < p < p_0(\rho_0)$ (c.f. \cite{R3}).

\begin{definition}\label{D:Dp} ({\rm $L^p$ Dirichlet problem, ${\mathcal
D}_p$}). We say that the Dirichlet problem is solvable for
$\mathcal{L}$ in $L^p$ on $D$ (or that $\mathcal{D}_p$ holds for
$\mathcal{L}$ on $D$), for $1 < p < \infty$, if ${\mathcal
C}{\mathcal D}$ holds for $\mathcal L$ and there is a constant $C$
(depending only on $\mathcal{L}$, $\lambda$, $n$, $D$ and $p$)
such that for all $g \in C(\partial D)$, the ${\mathcal
C}{\mathcal D}$ solution to \eqref{E:pde} (which we shall denote
by $u_g$) satisfies
\begin{equation}\label{E:Dp}
\norm{Nu_g}_{L^p(\partial D)} \le C \norm{g}_{L^p(\partial D)},
\end{equation}
where $N$ is the non-tangential maximal operator (see below). Unless explicitly
stated, the assumed measure on $\partial D$ is $\sigma$, standard surface measure.
\end{definition}

Here the non-tangential maximal operator $N$ (when necessary,
$N_{\alpha}$) is defined as
\[ N_{\alpha}u (Q) = \sup_{x\in\Gamma_{\alpha}(Q)}\abs{u(x)},\]
where $\Gamma_{\alpha}(Q)$ denotes a truncated cone interior to
$D$ of aperture $\alpha$ based at $Q \in \partial D$, i.e.,
\[ \Gamma_{\alpha}(Q) = \{ x \in D : \abs{x - Q} \le
(1+\alpha)\delta(x)\} \cap B_{r^*}(Q).\] Throughout, $ \alpha
>\alpha^*(D)> 0$, with $\alpha^*$ and $r^* >0$ determined by the
Lipschitz character of the domain and its size. Finally,
$\delta(x) := \dist{x, \partial D}$; $B_r(x)$ denotes the ball of
radius $r$ centered at $x$.\vspace*{2mm}

In this paper we consider two fundamental questions. The first one
is whether the $L^p$ solvability can be perturbed, that is, if
${\mathcal L}_0$ and ${\mathcal L}_1$ are two operators close in
some sense, under what conditions does the solvability of the $L^p$
Dirichlet problem for one operator imply the same for the other?
This question has a long history in the case of second order
elliptic divergence form operators. For our purposes the papers
\cite{D2} and \cite{FKP} are of particular importance. Operators
with first order terms are considered in \cite{HL}.

The non-divergence form case has a considerably shorter history
since new difficult issues arise. In particular, the non-uniqueness of
so-called {\it weak solutions} causes trouble in the most general case (see \cite{Nad} and \cite{Sa}).

However, assuming, as we do, that the coefficients of the
non-divergence form operators considered have a small BMO norm,
 one can establish the existence of \emph{strong solutions} (i.e. solutions in $W^{2,p}_{loc}(D)$, c.f
\cite{CFL}). These are the solutions we consider in this paper.

The papers \cite{R1} and \cite{R2} have made very good progress in
settling the question whether results that hold in the divergence
form case extend to non-divergence form operators. In particular, these
papers show that if the elliptic measure of an operator ${\mathcal
L}_0$ is in the Muckenhaupt $A_\infty(d\sigma)$ class, then so is the
elliptic measure of an operator ${\mathcal L}_1$ under the same
assumptions as in \cite{FKP}. This implies that
$L^p$ solvability of the operator ${\mathcal L}_0$ gives $L^q$
solvability of the operator ${\mathcal L}_1$ (for $q$ potentially much larger
than $p$). The paper \cite{R2} also considers first order terms
(drift terms).

In our Theorem \ref{T:perturbation} we settle the question of whether
$q$ can be taken to be the same as $p$, and the answer is affirmative if the
coefficients of the considered operators are sufficiently close in
the sense of Carleson measure. Analogous results for divergence
form operators have been established before (\cite{D2} and \cite{FKP}).
We do not consider first order terms as \cite{R2} does;
however, this can be done, as all the necessary ingredients are
already in place. We choose not to do it here to make our
already very technical exposition more readable.

The second fundamental question we settle here is the question of
finding a broad condition on coefficients of the non-divergence
form operator that guarantee $L^p$ solvability. Again, the case of
divergence form operators serves as a model. There are two
particulary important results to mention here. In \cite{KP} it was
established that under the assumption that $t\abs{\grad A}^2$ is a
Carleson measure, the $L^p$ Dirichlet problem
is solvable. In \cite{DPP} this condition was relaxed (the
gradient is replaced by an oscillation-type condition), and
it was also shown there that given $p\in (1,\infty)$ the $L^p$
solvability depends on the norm of the Carleson measure. If the
norm is small, the $L^p$ solvability for particular $p$ holds.
Moreover, since first order terms are also considered, the results
in \cite{KP} and \cite{DPP} do apply to non-divergence form
operators, but only under the stronger gradient condition
$t\abs{\grad A}^2$.

The missing piece to prove this under the weaker
``oscillation condition'' is a strong perturbation theorem for
non-divergence form operators which we establish here. Hence
the result in \cite{DPP} is substantially improved. This is
formulated in Theorem \ref{T:main}. As previously mentioned,
the weaker $A_\infty$ version of this result is already done
in \cite{R2}.\vspace*{2mm}

The structure of this paper is as follows: section 2
contains a few definitions needed to formulate two main results,
which is done in section 3. Section 4 expands the list of
definitions and introduces a few technical preliminaries. Section
5 contains the proof the the perturbation result and section 6
is dedicated to the $L^p$ solvability under the Carleson condition on
the coefficients of our operator. Finally, sections 7 and 8 contain
the proofs of two auxiliary lemmas.

\section{Basic definitions}\label{S:basic_definitions}

Given $f \in L^1_{loc}(\R^n)$, let
\[\eta(r,x) = \eta_f(r,x) = \sup_{s \le r} \frac{1}{\abs{B_s(x)}} \int_{B_s(x)} \abs{f(y)-f_{B_s(x)}} \, dy,\]
where $f_E$ is the average value of $f$ on $E$. Then $f \in
\BMO(\R^n)$ (i.e., $f$ has \emph{bounded mean oscillation}) if
$\eta \in L^{\infty}(\R^+,\R^n)$. Moreover, $\norm{f}_{\BMO} =
\norm{\eta_f}_{L^{\infty}(\R^+, \R^n)}$.

Let $\eta(r) = \norm{\eta(r, \cdot)}_{L^{\infty}(\R^n)}$. We say
$f \in \VMO(\R^n)$ ($f$ has \emph{vanishing mean oscillation}) if
$\lim_{r\to 0^+}\eta(r) = 0$. Finally, a function $f \in
\BMO_{\rho}(\R^n)$ if $\eta(r) \in \Phi(\rho)$, where $\Phi(\rho)$
is the collection of all non-decreasing functions $\eta:\R^+ \to
\R^+$ such that there exists a $\zeta > 0$ such that $\eta(r) \le
\rho$ for all $r < \zeta$. These spaces ($\BMO$, $\VMO$,
$\BMO_{\rho}$) can be restricted to a Borel set $G$ using standard
methods.\vspace*{2mm}

The setting for our work is a Lipschitz domain $D$. A bounded,
connected domain $D \in \R^n$ is called a \emph{Lipschitz domain}
if there is a finite collection $\{(I_i, \phi_i)\}$ of right
circular cylinders $I_i$ and Lipschitz functions $\phi_i$ ($\phi_i
: \R^{n-1} \to \R$, and there is an $L > 0$ such that for all $x,y
\in \R^{n-1}$, $\abs{\phi(x)-\phi(y)}\le L\abs{x-y}$) such that
the following hold:

\begin{enumerate}[(i)]
\item  The collection of cylinders $\{I_i\}$ covers the boundary,
$\partial D$, of $D$. \item  The bases of the cylinders have
positive distance from $\partial D$. \item  Corresponding to each
pair $(I_i,  \phi_i)$, there is a coordinate system $(x,s)$ with
$x \in \R^{n-1}$, $s \in \R$ such that the $x$-axis is parallel to
the axis of $I_i$ and such that $I_i \cap D = \{(x,s): s >
\phi_i(x)\}\cap I_i$ and $I_i \cap \partial D = \{(x,s):
s=\phi_i(x)\} \cap I_i$.
\end{enumerate}

Without loss of generality, we will assume that $D$ is containted
within the unit ball centered at the origin of $\R^n$ and that $D$
contains the origin, i.e., we assume $D \subset B_1(0)$ with $0
\in D$.

For points $Q \in \partial D$ and $r > 0$, we denote the boundary
ball of radius $r$ at $Q$ by $\Delta_r(Q) = B_r(Q) \cap \partial
D$. The \emph{Carleson region} $T_r(Q)$ above $\Delta_r(Q)$ is
given by $T_r(Q) = B_r(Q) \cap D$. We say that a measure $\mu$ on
$D$ is a \emph{Carleson measure} if there is an $M < \infty$ such
that  \[\sup_{r>0, Q \in \partial
D}\frac{\mu(T_r(Q))}{\sigma(\Delta_r(Q))} = M.\]

\section{Main results}

The following perturbation theorem is modelled on Dahlberg's
theorem (Theorem 1 in \cite{D2}, re-proven as Theorem 2.18 in
\cite{FKP}).

\begin{theorem}\label{T:perturbation} Consider operators $\mathcal{L}_0$, $\mathcal{L}_1$, with
$\mathcal{L}_k = a^{ij}_k(x) \partial_{ij}$ on a Lipschitz domain
$D$, $\epsilon(x) = \left(a^{ij}_0(x) - a^{ij}_1(x)\right)_{i,j}$
and $\mathbf{a}(x) = \sup_{z \in B_{\frac{\delta(x)}{2}}(x)}
\abs{\epsilon(z)}$. Let $\lambda>0$ be the ellipticity constant of
the operator $\mathcal{L}_0$, and let
\begin{equation}
\sup_{Q \in \partial D, r > 0} \frac{1}{\sigma(\Delta_r(Q))}
\int_{T_r(Q)} \frac{\mathbf{a}^2(x)}{\delta(x)} \, dx=\epsilon_0 <
\infty.\label{E:perturbation}
\end{equation}
Assume that the $L^p$ Dirichlet problem is solvable for the operator
$\mathcal{L}_0$ with a constant $C_p>0$ in the estimate
\eqref{E:Dp} for some $1<p<\infty$.

There exist constants $\rho_0>0$ (independent of $p$) and
$M=M(p,D,\lambda,C_p,\rho)>0$ such that if $a^{ij}_0\in\BMO_{\rho}$ with $\rho<\rho_0$, and if
$\epsilon_0<M$, then the $L^p$ Dirichlet problem is solvable for
the operator $\mathcal{L}_1$.
\end{theorem}

\begin{remark} This theorem is a direct improvement of
Theorem 1.1 in \cite{R1} and Theorem 2.1 in \cite{R2} where a
statement of the type $\mathcal{D}_p$ for ${\mathcal L}_0$
$\Longrightarrow$ $\mathcal{D}_q$ for ${\mathcal L}_1$ was
established with $q>>p$.
\end{remark}

\begin{remark} It suffices to assume that the condition
$\epsilon_0<M$ in Theorem \ref{T:perturbation} only holds for all
Carleson regions $T_r$ such that $r\le r_0$ for some $r_0>0$.
This is due to the comparability of the elliptic measures of two operators
whose coefficients are the same near the boundary (see Lemma 2.15 of
\cite{R1}).
\end{remark}

\begin{remark} The theorem can be formulated on more
general domains. In fact, we never explicitly use that the
boundary of $D$ has a graph-like structure. The minimal geometric
structure needed is that $D$ be a chord-arc domain and
non-tangetially accessible.
\end{remark}

\begin{remark} The number $\rho_0$ is chosen such that
if $a^{ij}\in \BMO_{\rho}$ for any $\rho<\rho_0$, $\mathcal{C}\mathcal{D}$ holds for the operator
$\mathcal{L}=a^{ij}\partial_{ij}$ on $D$ and so that the value of $p_0(\rho_0)$ referenced in the discussion after Definition \ref{D:CD} satisfies $p_0(\rho_0) > 2$.
\end{remark}

\begin{theorem}\label{T:main} Let $1<p<\infty$, let $0<\lambda<\infty$ be a fixed ellipticity constant, and let $D$ be a
Lipschitz domain with Lipschitz constant $L$. Let ${\mathcal
L}=a^{ij}\partial_{ij}$ be an elliptic operator with ellipticity
constant $\lambda$.

If \[ \sup \left\{ \frac{\abs{a^{ij}(x) -
    \avg{a^{ij}(z)}}^2}{\delta(x)} : x \in
B_{\frac{\delta(z)}{2}}(z)\right\} \] is the density of a Carleson
measure in $D$ with Carleson constant $M$, then there is a
constant $C(p,\lambda)
> 0$ such that if $L < C(p,\lambda)$ and $M < C(p,\lambda)$,
the Dirichlet problem  $\mathcal{D}_p$ is solvable for the
operator $\mathcal{L}$. [Here $\avg{a^{ij}(z)}$ is the average of
the coefficient $a^{ij}$ over the ball $B_{\delta(z)/2}(x)$.]
\end{theorem}

This is a substantial improvement of Corollary 2.5 in
\cite{DPP}, in the spirit of Corollary 2.3 in the same paper for
divergence form operators. This also improves Theorem 2.4 in
\cite{R2}. Note that we do not consider drift terms as \cite{R2}
does.

\section{Notation and Technical Preliminaries}

To enhance the readability of this paper, we have kept our notation the
same as in \cite{R1}. We rely heavily on certain results from \cite{R1} in the
technical part of this paper.

Throughout the paper, we use $A \ls B$ to mean there is a constant
$C$, depending on, at most, $n, \lambda, \eta$ and $D$ such that
$A \le C B$; similarly for $A \gs B$. If $A \ls B$ and $A \gs B$,
then we say $A \eqs B$.

Let $\mathcal L$ be an elliptic operator for which $\mathcal{CD}$
 holds. By the maximum principle, the mapping $g
\mapsto u_g(x)$ is a positive linear functional on $C(\partial D)$
for each fixed $x \in D$. The Riesz representation theorem then
gives a unique regular positive Borel measure $\omega^x$ on
$\partial D$ such that
\[ u(x) = \int_{\partial D} g(Q) \, d\omega^x(Q). \]
This measure is called the harmonic measure for $\mathcal{L}$ on
$D$.

Given a non-decreasing function $\eta$, we denote by $O(\lambda,
\eta)$ the class of operators $\mathcal{L} = a^{ij}\partial_{ij}$
with symmetric coefficients satisfying \eqref{E:ellipticity} such
that $a^{ij} \in \BMO(\R^n)$ with BMO modulus of continuity $\eta$
in $D$. We use $O(\lambda)$ if there is no restriction on the
regularity of $\mathcal{L}$.

The theory of weights plays an important role in what follows. Given a $p$, $1
< p < \infty$, and two measures $\mu$ and $\nu$ on $\partial D$, if $\mu$ is
absolutely continuous with respect to $\nu$, let $k = \frac{\partial
\mu}{\partial \nu}$. Then, we say that $\mu \in A_p(d\nu)$ if there is a
constant $C < \infty$ such that for all boundary balls $\Delta$ (i.e., for some
$r> 0$, $Q \in \partial D$, $\Delta = \Delta_r(Q)$),
\begin{equation}\label{E:Ap_defn}
\left( \frac{1}{\nu(\Delta)} \int_{\Delta}k \,
d\nu\right)\left(\frac{1}{\nu(\Delta)} \int_{\Delta} k^{-\frac{1}{p-1}} \,
d\nu\right)^{\frac{1}{p-1}} \le C.
\end{equation}
We say that $\mu \in RH_p(d\nu)$, the reverse-H\"older class, if there is a
constant $C$ such that for all boundary balls $\Delta$,
\begin{equation}\label{E:RH_defn}
\left(\frac{1}{\nu(\Delta)} \int_{\Delta} k^p \, d\nu \right)^{\frac{1}{p}} \le
C \frac{1}{\nu(\Delta)} \int_{\Delta} k \, d\nu.
\end{equation}

Note that $\mu \in A_p(d\nu)$ if and only if $\nu \in RH_{p'}(d\mu)$, with $p' =
\frac{p}{p-1}$. The best constant $C$ in \eqref{E:Ap_defn} is called the $A_p$
``norm'' of $\mu$ and is denoted $A_p(\mu| d\nu)$. Recall that the
assumed measure on $\partial D$ is $\sigma$, standard surface measure, so by
$\mu \in A_p$, we mean $\mu \in A_p(d\sigma)$. Also, these classes of measures
(or weights) are related:
\[ \bigcup_{p'>1} RH_{p'}(d\nu) = \bigcup_{p > 1} A_p(d\nu) =:
A_{\infty}(d\nu).\]

A crucial ingredient in what follows is the fundamental theorem
relating weights to solutions of elliptic partial differential
equations (first proved by Dahlberg in \cite{D2}):
\begin{theorem}
Let $\omega$ be the harmonic measure with respect to $\mathcal{L}$ on $D$, and
let $\mu$ be a Borel measure on $\partial D$. Then the following are
equivalent:
\begin{enumerate}[(i)]
\item $\omega \in A_{\infty}(d\mu)$.
\item There is a $1 < p < \infty$ such that $\mathcal{D}_p$ holds, i.e,
\[\norm{Nu_g}_{L^p(\partial D, d\mu)} \le C_p \norm{g}_{L^p(\partial D,
d\mu)}.\]
\item $\omega$ is absolutely continuous with respect to $\mu$ and $\omega \in
RH_{p'}(d\mu)$ (where, again, $p' = \frac{p}{p-1}$).
\end{enumerate}
\end{theorem}

\begin{lemma}\label{L:adjoint}(Theorem 2.5 in \cite{R1})
If we let $p \in [n, \infty)$, $w \in A_p$, then there is a
constant $\rho_p$ such that if $\eta \in \Phi(\rho_p)$ and
$\mathcal{L} \in O(\lambda, \eta)$, for any $f \in L^p(D,w)$,
there exists a unique $u \in C(\overline{D}) \cap W^{2,p}(D,w)$
such that $\mathcal{L}u = f$ in $D$ and $u = 0$ on $\partial D$.
\end{lemma}

Then, with $f$ and $u$ as in Lemma \ref{L:adjoint}, for each $x \in D$, the
maximum principle implies that the positive linear functional $f \to -u(x)$ is
bounded on $L^p(D)$. The Riesz representation theorem then gives us the unique
non-negative function $G(x,\cdot) \in L^{p'}(D)$ ($p' = \frac{p}{p-1}$) such
that
\[u(x) = - \int_D G(x,y)f(y) \, dy.\]
This is the Green's function for $\mathcal{L}$ in $D$.

\begin{definition}\label{D:adjoint_solution}
Given $\mathcal{L} \in O(\lambda)$, $v \in L^1_{loc}(D)$ is an \emph{adjoint
solution} of $\mathcal{L}$ in $D$ if
\[\int_D v \mathcal{L}\phi \, dx = 0\] for all $\phi \in C^{\infty}_c(D)$. In this case, we write $\mathcal{L}^*v = 0$.
\end{definition}

Recall that we are assuming $D \subset B_1(0)$. Pick a point $\bar{x} \in \partial B_{9}(0)$, and let
\begin{equation}\label{E:nas}
\wp(y) = G_{\mathcal{L}, B_{10}(0)}(\bar{x},y) \text{ in $B_{10}(0)$,}
\end{equation}
where $G_{\mathcal{L}, B_{10}(0)}$ is the Green's function for
$\mathcal{L}$ in $B_{10}(0)$. The following technical estimate is
quite useful.
\begin{lemma}\label{L:Greens_est} (Lemma 2 in \cite{EK}). Let $G(x,y)$ be the Green's function in
$D$ for $\mathcal{L} \in O(\lambda)$. Then there is a constant $r_0$ depending
on the Lipschitz character of $D$, such that for all $Q \in \partial D$, $r \le
r_0$, $y \in \partial B_r(Q) \cap \Gamma_1(Q)$ and $x \notin T_{4r}(Q)$, the
following holds:
\[\frac{G(x,y)}{\delta(y)^2} \frac{\wp(B(y))}{\wp(y)} \eqs
\omega^x(\Delta_r(Q)),\] with $\wp$ as defined in \eqref{E:nas},
$B(y)=B_{\delta(y)/2}(y)$ and $\wp(B(y))=\int_{B(y)}\wp(y)dy$.
\end{lemma}

Following Rios, \cite{R1}, we define a modified area function and
non-tangential maximal function which are adapted for the
non-divergence form situation.

\begin{definition}\label{D:area_function}[Area functions]
For a function $u$ defined on $D$, the area function of aperture $\alpha$,
$S_{\alpha} u$, and the second area function of aperture $\alpha$, $A_{\alpha}
u$, are defined as
\begin{align*}
S_{\alpha}u(Q)^2 &= \int_{\Gamma_{\alpha}(Q)} \frac{\delta^2(x)}{\wp(B(x))}
\abs{\grad u(x)}^2 \wp(x) \, dx, \text{ and} \\
A_{\alpha}u(Q)^2 &= \int_{\Gamma_{\alpha}(Q)} \frac{\delta^4(x)}{\wp(B(x))}
\abs{\grad^2 u(x)}^2 \wp(x) \, dx,
\end{align*}
with $\wp(x)$ as in \eqref{E:nas}, $B(x) = B_{\delta(x)/2}(x)$, and $Q \in
\partial D$.
\end{definition}
We also recall Rios' modified non-tangential maximal function,
\begin{equation}\label{E:ntmax}
\left(\widetilde{N}_{\alpha}(v)\right)^2(Q) := \sup_{x \in
\Gamma_{\alpha}(Q)} \int_{B_0(x)} v(y)^2 \frac{\wp(y)}{\wp(B(y))}
\, dy.
\end{equation}
Here, $B_0(x)$ denotes $B(x, \frac{\delta(x)}{6})$.

\section{Proof of the perturbation theorem \ref{T:perturbation}}

The structure of our proof owes much to the proof of Theorem 2.18
in \cite{FKP}. We use $S_\alpha$,
$A_\alpha$ and $\widetilde{N}_\alpha$ as defined in the section
above.

Let ${\mathcal L}_0$ and ${\mathcal L}_1$ be two operators as in
Theorem \ref{T:perturbation}, and consider any continuous boundary
data $g$. We first establish that $\mathcal{CD}$ holds
for ${\mathcal L}_1$. We observe that
\mbox{$\norm{A_0-A_1}^2_{L^\infty(D)}\ls \epsilon_0$.} Since
$a^{ij}_0\in \BMO_{\rho}$ and \eqref{E:perturbation} holds with
$\epsilon_0 < M$, $a^{ij}_1\in \BMO_{\rho+\epsilon}$, where
$\epsilon$ can be arbitrary small (it depends on $M$). So if $M$
is made small enough, we can ensure that
$\rho+\epsilon<\rho_0$. Recall that $\rho_0$ is chosen to guarantee that
$\mathcal{C}\mathcal{D}$ holds for any operator
$\mathcal{L}_k=a^{ij}_k\partial_{ij}$, as long as $a^{ij}_k \in \BMO_{\rho}$
 and $\rho < \rho_0$.

Notice also that if $\lambda$ is the ellipticity constant of
${\mathcal L}_0$, one can guarantee that the ellipticity constant
of ${\mathcal L}_1$ stays bounded away from zero, say by
$\lambda/2$, by making $M$ smaller if necessary.

Hence we can talk about solutions $u_0$ and $u_1$ to the
corresponding Dirichlet problem with the same boundary data $g$
for $\mathcal{ L}_0$ and $\mathcal{ L}_1$, respectively. Let
\mbox{$F=u_0-u_1$.} If follows that $\mathcal{L}_0 F=-\mathcal{L}_0 u_1$,
 so
\[F(x)=\int_{D} G_0(x,y)\mathcal{ L}_0u_1\,dy.\]
Here $G_0$ is the Green's function of the operator $\mathcal{ L}_0$.

We will use the following two lemmas and defer their proofs until
later. The following is analogous to Lemma 2.9 in \cite{FKP}.
\begin{lemma}\label{L:2.9}
There exists a constant $C = C(\lambda, n)$ such that under the
hypotheses of Theorem \ref{T:perturbation},
\[ \widetilde{N}F(Q) + \widetilde{N}(\delta\abs{\grad F})(Q) \le C
\epsilon_0 M_{\omega_0}(A_{\tilde{\alpha}}u_1)(Q).\]
\end{lemma}

The next lemma is analogous to Lemma 2.16 in \cite{FKP}.
\begin{lemma}\label{L:good-lambda} Let $\alpha>0$. Then there exists
$0<\beta<\alpha$ depending only on the dimension, the number $\alpha$
and the Lipschitz constant of the domain $D$ such that the
following holds:

Suppose that $S_\beta(F)(P)\le\lambda$ for some $P$ in a surface
ball $\Delta=\Delta(P_0,r)\subset\partial D$. Then there exists
$c>0$, $\delta>0$ depending only on the Lipschitz character of the
domain $D$ and the ellipticity constant of the operator
$\mathcal{L}_0$  such that for any $\gamma>0$
\begin{eqnarray}
\sigma\big(\{Q\in\Delta&;&S_\beta(F)>2\lambda,\,
\widetilde{N}_\alpha(F)\le\gamma\lambda,\widetilde{N}_\alpha(\delta\abs{\grad
F})\le\gamma\lambda,\nonumber\\&&
\widetilde{N}_\alpha(F)A_\alpha(u_1)\le
(\gamma\lambda)^2\}\big)\le
c\gamma^{\delta}\sigma(\Delta).\label{GL1}
\end{eqnarray}
\end{lemma}

Assuming Lemmas \ref{L:2.9} and \ref{L:good-lambda}, we have
\begin{align*}
\int_{\partial B_1}\widetilde{N}(F)^p \, d\sigma & \le \int_{\partial B_1}\left(\widetilde{N}(F)^p + \widetilde{N}(\delta \abs{\grad F})^p\right) \, d\sigma \\
& \le C \epsilon_0 \int_{\partial B_1} (M_{\omega_o}(A_{\tilde{\alpha}}u_1))^p \, d\sigma \\
& \le C \epsilon_0 \int_{\partial B_1} (M_{\omega_o}(A_{\tilde{\alpha}}u_1))^p k_0^{-1} \, d\omega_0\\
& \le C' \epsilon_0 \int_{\partial B_1} A_{\tilde{\alpha}}(u_1)^p
k_0^{-1} \, d\omega_0,
\end{align*}
where $k_0 = \frac{d\sigma}{d\omega_0}$. Recall that $k_0 \in
RH_{p'}(d\sigma)$ (since we are assuming ${\mathcal D}_p$
solvability for ${\mathcal L}_0$) is equivalent to $k_0^{-1} \in
A_p(d\omega_0)$.

We then have
\begin{eqnarray}
\int_{\partial B_1} \left(\widetilde{N}(F)^p + \widetilde{N}(\delta \abs{\grad F})^p\right) \, d\sigma &\le& C \epsilon_0 \int_{\partial B_1} A_{\tilde{\alpha}}(u_1)^p \, d\sigma \nonumber\\
& \le& C' \epsilon_0 \int_{\partial B_1} S_{c\widetilde{\alpha}}(u_1)^p \, d\sigma \nonumber\\
& \le& C'' \epsilon_0 \int_{\partial B_1} S_{{\beta}}(u_1)^p \, d\sigma \label{eq14}\\
& \le& C''' \epsilon_0 \int_{\partial B_1} (S_{{\beta}}(u_0)^p +
S_{{\beta}}(F)^p) \, d\sigma,\nonumber
\end{eqnarray}
where we used Theorem 2.19 from \cite{R1} and also Theorem 2.17 of
\cite{R1} for \\
\mbox{$\norm{S_{\tilde{c\alpha}}(u_1)}_{L^p}\eqs\norm{S_{{\beta}}(u_1)}_{L^p}$.}
It remains to deal with these terms.

First, we note  that $\int_{\partial B_1}
\left(S_{c\tilde{\alpha}}(u_0)\right)^p\, d\sigma \ls
        \int_{\partial B_1} f^p \, d\sigma$, using Theorem 2.17 from \cite{R1} and our assumption that the Dirichlet problem is solvable for $\mathcal{L}_0$.

For the second term, we will use the good-lambda inequality from
Lemma \ref{L:good-lambda}, more specifically its Corollary
\ref{c1}. According to it, we have an estimate
\begin{equation}
\int_{\partial B_1} S_{{\beta}}(F)^p \, d\sigma\le 2C
\int_{{\partial B_1}}
\left(\widetilde{N}(F)^p+\widetilde{N}(\delta \abs{\grad F})^p
\right)d\sigma + \int_{{\partial
B_1}}S_{\beta}(u_0)^p\,d\sigma.\label{eq15}
\end{equation}
The term $\int_{{\partial B_1}}S_{\beta}(u_0)^p\,d\sigma$ is harmless
and can be estimated by $\int_{\partial B_1} f^p \, d\sigma$, as
above. Now we put \eqref{eq14} and \eqref{eq15} together to obtain

$$\int_{\partial B_1} \left(\widetilde{N}(F)^p + \widetilde{N}(\delta
\abs{\grad F})^p\right) \, d\sigma \le C \epsilon_0
\int_{\partial B_1} \left(\widetilde{N}(F)^p +
\widetilde{N}(\delta \abs{\grad F})^p + f^p\right) \, d\sigma.$$

Hence for $\epsilon_0$ sufficiently small so that $C
\epsilon_0\le 1/2$, we have that
$$\int_{\partial B_1} \left(\widetilde{N}(F)^p + \widetilde{N}(\delta
\abs{\grad F})^p\right) \, d\sigma \le C \int_{\partial B_1} f^p
\, d\sigma.$$

This is the estimate required for the solvability of the Dirichlet
problem ${\mathcal D}_p$ for ${\mathcal L}_1$, as $u_1=u_0-F$, and
for $u_0$ we have the needed estimates due to the $L^p$
solvability for ${\mathcal L}_0$. Since for $\alpha'<\alpha$ we
have a pointwise estimate:
\[ N_{\alpha'} u_1(Q) \ls \widetilde{N}_\alpha u_1(Q) + \widetilde{N}_\alpha(\delta\abs{\grad u_1})(Q), \]
the theorem follows.

\section{Proof of Theorem \ref{T:main}}
Now that we have established the perturbation theorem, we can
easily dispense with the proof of Theorem \ref{T:main}. The first
part of the proof, dealing with the smooth perturbation of $A$, is
exactly the same as the smooth perturbation part of the proof of
Corollary 2.3 in \cite{DPP}. We repeat it here for convenience.

We prove this in the flat case; the general result will follow
from a change of variables by Ne\v{c}as and Stein (see, e.g.,
p.2 of \cite{DPP} for details). The notation $\avg{a}$ at a point
$(y,s)$ represents the average of $a$ over the ball
of radius $s/2$ centered at $(y,s)$ (denoted $B_{s/2}(y,s)$). Given a matrix coefficient
$a(x,t)$ in $\R^n_+$, set $\tilde{a}(x,t) = \int
a(u,s)\phi_t(x-u, s-t) \, ds \, du$, where $\phi$ is a smooth bump
function supported in the ball of radius $1/2$ and $\phi_t(y,s) =
t^{-n}\phi(y/t, s/t)$.

We are assuming that
\begin{equation}
\left(\sup\left\{\abs{a(y,s) - \avg{a(x,t)}}^2:(y,s) \in
B_{t/2}(x,t)\right\}\right) \frac{dx \, dt}{t}\label{E:carl22}
\end{equation}
is a Carleson measure with small norm.

We aim to establish three facts:
\begin{equation}\label{E:smoothCarleson}
t\abs{\grad\tilde{a}(x,t)}^2 \, dx\, dt
\end{equation} is a Carleson measure with small norm,
\begin{equation}\label{E:smoothperturbation}
\left(\sup\left\{\abs{a(y,s) - \tilde{a}(y,s)}^2:(y,s) \in
B_{t/2}(x,t)\right\}\right)\frac{dx\, dt}{t}
\end{equation}
satisfies the hypotheses of Theorem \ref{T:perturbation}, and
\begin{equation}\label{E:BMOrho}
a(x,t)\in\BMO_{\rho} \text{ for }
\rho=\rho(M), \text{ with } \rho\to 0 \text{ as } M\to 0.
\end{equation}

Given the results in \cite{DPP}, the condition
(\ref{E:smoothCarleson}) implies that $\mathcal{D}_p$ holds for
the operator with coefficients $\widetilde{A}$. Using \eqref{E:BMOrho}, if $M$ is chosen
sufficiently small we will have $\rho<\rho_0$. Combining this with \eqref{E:smoothperturbation}, as in the proof of Theorem \ref{T:perturbation} above, yields that $\widetilde{A}$ is in $\BMO_{\tilde{\rho}}$, for some $\tilde{\rho} < \rho_0$. Thus, the hypotheses for Theorem \ref{T:perturbation} are satisfied and, therefore, $\mathcal{D}_p$ holds for the operator with coefficients $A$.\vspace*{2mm}

That \eqref{E:smoothCarleson} follows from the hypotheses of Theorem \ref{T:main} is a straightforward calculation; apply the gradient to $\phi_t(y,s)$ and subtract a constant from the $a_{ij}$ inside the integrand to see that
\[ \abs{\grad\tilde{a}(x,t)} \le C t^{-1}\left(\sup\left\{\abs{a(y,s)- \avg{a(x,t)}}: (y,s) \in B_{t/2}(x,t)\right\}\right).\]

The proof of \eqref{E:smoothperturbation} is equally
straightforward; add and subtract the constant $\avg{a(x,t)}$
inside the difference. For precise details see \cite{DPP} and a
similar calculation in \cite{R1}.

It only remains to prove \eqref{E:BMOrho}. Choose an arbitrary point $(x,t)$ in our domain. We shall check
that the function $a$ is $\BMO$ near this point. Consider a ball
$B$ of radius $s>0$ centered at the point $(x,t)$. There are three
cases to consider:
\begin{enumerate}[(i)]
\item a small ball, with $s<t/2$,
\item a large ball, with $s\ge 2t$,
\item an intermediate ball, with  $t/2\le s<2t$.
\end{enumerate}

As we shall see, only the cases (i) and (ii) are fundamental. Case (iii)
is merely a combination of the approaches taken in (i) and (ii).

In case (i), \eqref{E:carl22} trivially gives that
\[
\sup_{(y,u)\in B_{t/8}(x,t)}\abs{a(y,u) - \avg{a(x,t)}}^2 \ls M,
\]
hence
\[ \text{osc}_{{B_{t/8}(x,t)}} a =\max_{ij}\abs{\sup_{(y,u)\in
B_{t/8}(x,t)}a^{ij}(y,u) - \inf_{(y,u)\in B_{t/8}(x,t)}
a^{ij}(y,u)} \ls M^{1/2}, \] From this
\begin{equation}
\text{osc}_{{B_{s}(x,t)}} a \le \text{osc}_{{B_{t/2}(x,t)}} a \ls
M^{1/2} \label{E:osc}
\end{equation} as the ball or radius $t/2$ can be
covered by a fixed number (depending only on dimension) of balls
of radius $t/8$. This immediately gives

\[\abs{B_s(x,t)}^{-1}\int_{B_s(x,t)}\abs{a^{ij}(y,u)-\avg[B_u]{a^{ij}}}\, dy\, du\ls
M^{1/2},\] hence $\rho\le CM^{1/2}$.

If (ii) holds then $B_s(x,t)$ intersects the boundary $\{t=0\}$ at
a large set of area of order $s^{n-1}$. One might think
of $D\cap B_s(x,t)$ as a subset of a larger Carleson box
$T(\Delta)$, where $\Delta$ is a surface ball on the boundary
$\{t=0\}$ of radius comparable to $s$ (a multiple of $s$ where the
constant depends on the dimension of our domain). Therefore, it will
suffice to prove that
\[\int_{T(\Delta)} \abs{a^{ij}(x,t)-\avg[T(\Delta)]{a^{ij}}}\, dx\, dt\ls M^{1/2}s^{n},\]
from which again $\rho\le CM^{1/2}$ on such balls.

In fact, the exact average that gets subtracted off in the $\BMO$
norm does not matter, so we might as well prove that
\[\int_{T(\Delta)} \abs{a^{ij}(x,t)-\avg[T_0(\Delta)]{a^{ij}}}\, dx \, dt\ls M^{1/2}s^{n}.\]
Here $T_k(\Delta)$, $k=0,1,2,\dots$ is defined diadically by
\[T_k(\Delta)=\{(x,t)\in T(\Delta);\, t\in (2^{-k-1}s,2^{-k}s]\}.\]
It follows that $T(\Delta)$ is a disjoint union of $T_k(\Delta)$,
$k\ge 0$.

By \eqref{E:osc} we immediately get that
\[\abs{a^{ij}(x,t)-\avg[T_0(\Delta)]{a^{ij}}}\ls M^{1/2},\qquad\text{for all }(x,t)\in T_0(\Delta).\]
Now consider $(x,t)\in T_1(\Delta)$. By using \eqref{E:osc} twice
we get that
\[\abs{a^{ij}(x,t)-\avg[T_0(\Delta)]{a^{ij}}}\ls 2M^{1/2},\qquad\text{for all }(x,t)\in T_1(\Delta),\]
and inductively
\[\abs{a^{ij}(x,t)-\avg[T_0(\Delta)]{a^{ij}}}\ls (k+1)M^{1/2},\qquad\text{for all }(x,t)\in T_k(\Delta).\]
Hence,
\begin{align*}
\int_{T(\Delta)}
&\abs{a^{ij}(x,t)-\avg[T_0(\Delta)]{a^{ij}}}\, dx\, dt = \\
&\qquad\sum_{k=0}^\infty
\int_{T_k(\Delta)}
\abs{a^{ij}(x,t)-\avg[T_0(\Delta)]{a^{ij}}}\, dx\, dt\\
&\qquad\ls \sum_{k=0}^\infty (k+1)M^{1/2}(2^{-k}s^n)\eqs M^{1/2}s^n,
\end{align*}
since $\abs{T_k(\Delta)}\eqs 2^{-k}s^n$.

Case (iii) is a combination of these two approaches where one
considers the integrals on pieces $B_{s}(x,t)\cap\{(y,u);\, u\in
(2^{-k-1}t,2^{-k}t]\}$. We leave the details for the reader.

By combining (i), (ii) and (iii) we see that
$a^{ij}\in\BMO_{\rho}$ for $\rho\le CM^{1/2}$, where $M$ is the bound on the
Carleson measure of \eqref{E:carl22} and $C>0$ is a constant that
depends on the dimension of our domain $D$.

\section{Proof of Lemma \ref{L:2.9}}

\begin{proof}[Proof of Lemma \ref{L:2.9}] From \cite{R1}, Lemma
3.2, we know that
\begin{equation}
\tilde{N}_{\alpha}F(Q) \ls \epsilon_0 M_{\omega_0}(A_{\tilde{\alpha}}u_1)(Q)
\end{equation}
for some $\tilde{\alpha}$ slightly larger than $\alpha$. We
will also show that
\begin{align}\label{E:delta-grad-F-est}
&\left(\widetilde{N}_{\alpha}(\delta \abs{\grad F})\right)^{2}(Q) \\
& \qquad \ls
\widetilde{N}_{\tilde{\alpha}}(F)(Q)\widetilde{N}_{\tilde{\alpha}}(\delta\abs{\grad
F})(Q) + \left(\widetilde{N}_{\tilde{\alpha}}(F)\right)^2(Q) + \epsilon_0
\widetilde{N}_{\tilde{\alpha}}(F)(Q)A_{\tilde{\alpha}}(u_1)(Q).\notag
\end{align}
Combining these two yields the lemma. Thus it remains to show
\eqref{E:delta-grad-F-est}.

To this end, we fix $Q \in \partial D$, $x \in
\Gamma_{\alpha}(Q)$. Also, find the required value for $r_0$ in
Lemma 2.14 of \cite{R1} (if necessary, making $r_0 <
\frac{1}{4}$), and then choose $r^* \le r_0/2$, where $r^*$ is the
truncation level of $\Gamma_{\alpha}(Q)$.

Under these assumptions, if we take $y\in B_0=B(x,
\frac{\delta(x)}{6})$ then $y\in \Gamma_{\tilde{\alpha}}(Q)$ for
a slightly larger cone and also $y \in
\partial B_r(Q)$ for $r \le r_0$. Hence, Lemma 2.14
of \cite{R1} can be applied to all of the points in our integral. Lemma
2.14 in \cite{R1} provides the estimate:
\[
\frac{G_0(0, y)}{\omega_0(\Delta_r(Q))} \simeq \frac{\delta^2(y)
\wp(y)}{\wp(B(y))},
\]
where $r = \abs{y - Q}$. For $y \in B_0$, we have that
$\frac{5}{6}\delta(x) \le r \le \frac{7}{6}\delta(x)$. We observe
that for $r$ in this range, all values of $\omega_0(\Delta_r(Q))$
are comparable to the value of $\omega_0(\Delta_{\delta(x)}(Q))$,
as the measure is doubling. Also, let $\delta := \delta(x)$.

Following \cite{FKP}, we start with
\[
\int_{B_0} \delta^2(y) \abs{\grad F(y)}^2 \frac{\wp(y)}{\wp(B(y))}
\, dy \ls \frac{1}{\delta} \int_{\delta/6}^{\delta/5}
\int_{B_r} \delta^2(y) \abs{\grad F(y)}^2 \frac{\wp(y)}{\wp(B(y))}
\, dy \, dr,
\]
where $B_r=B(x,r)$. However, we need to average twice. Thus, we
estimate
\begin{equation}\label{E:ntmax_dgradF_est}
\int_{B_0} \abs{\grad F(y)}^2 \frac{\delta^2(y)\wp(y)}{\wp(B(y))} \, dy
\ls \frac{1}{\delta^2} \int_0^{\delta} \int_{\rho_1(s)}^{\rho_2(s)}
\int_{B_r} \abs{\grad F(y)}^2 \frac{\delta^2(y)\wp(y)}{\wp(B(y))} \, dy \, dr
\, ds,
\end{equation}
with $\rho_1(s) = (\beta_1 - \frac{1}{6})s + \frac{\delta}{6}$,
$\rho_2(s) = (\beta_2 - \frac15)s + \frac{\delta}5$, with $\beta_1
< \frac{1}{6} < \frac15 < \beta_2$. The $\beta_i$'s are yet to be
determined.

Then,
\begin{align*}
\int_{B_r} \delta^2(y) \abs{\grad F(y)}^2 \frac{\wp(y)}{\wp(B(y))} \, dy &
\simeq
\int_{B_r} \abs{\grad F(y)}^2 \frac{G_0(0, y)}{\omega_0(\Delta_r(Q))} \, dy \\
& \ls \frac{1}{\omega_0(\Delta_{\delta}(Q))} \int_{B_r}
A_0\grad F \cdot \grad F \, G_0(0,
y) \, dy \\
& \simeq \frac{1}{\omega_0(\Delta_{\delta}(Q))} \int_{B_r}
(\mathcal{L}_0(F^2) - 2F\mathcal{L}_0 F) \, G_0(0,
y) \, dy \\
&:= I_1 + I_2.
\end{align*}

Here $A_0$ is the matrix of coefficients $(a^{ij}_0)$. We first
estimate the contribution to \eqref{E:ntmax_dgradF_est} by $I_1$.
Integration by parts twice yields:
\begin{align*}
&I_1 \omega_0(\Delta_{\delta}(Q)) = \int_{B_r} \mathcal{L}_0(F^2)
\, G_0(0,
y) \, dy \\
&\quad = \sum \int_{\partial B_r} G_0 a_0^{ij}\partial_j(F^2) \nu_i \, d\sigma - \sum \int_{B_r} \partial_i(G_0 a_0^{ij}) \partial_j(F^2) \, dy \\
&\quad = \sum \int_{\partial B_r} \left(G_0 a_0^{ij} \partial_j(F^2)
\nu_i - \partial_i(a_0^{ij}G_0)F^2 \nu_j\right) \, d\sigma +
\int_{B_r} \mathcal{L}_0^*(G_0) F^2 \, dy.
\end{align*}
However, $\mathcal{L}_0^*(G_0) = 0$, so we are only left with
the two boundary terms. Hence,
\begin{align*}
    &\frac{1}{\delta} \int_{\rho_1(s)}^{\rho_2(s)} I_1 \, dr \\
    & \quad = \frac{1}{\delta \omega_0(\Delta_{\delta}(Q))} \sum \left[ \int_{B_{\rho_2(s)}\backslash B_{\rho_1(s)}} \left(G_0 a_0^{ij} \partial_j(F^2) \nu_i - \partial_i(a_0^{ij}G_0)F^2 \nu_j\right) \, dy\right] \\
    & \quad = \frac{1}{\delta \omega_0(\Delta_{\delta}(Q))} \sum \Bigg[ \int_{B_{\rho_2(s)}\backslash B_{\rho_1(s)}} G_0 a_0^{ij} \left(\partial_j(F^2)\nu_i + \partial_i(F^2)\nu_j + F^2 \partial_i \nu_j\right) \, dy \\
    & \qquad \qquad \qquad - \int_{\partial(B_{\rho_2(s)}\backslash B_{\rho_1(s)})} G_0 a_0^{ij} F^2 \nu_i \nu_j \, d\sigma \Bigg]
\end{align*}

We estimate this term by term, starting with
\begin{align*}
& \frac{1}{\delta \omega_0(\Delta_{\delta}(Q))} \sum \int_{B_{\rho_2(s)}\backslash B_{\rho_1(s)}} \abs{G_0 a_0^{ij} \left(\partial_j(F^2)\nu_i + \partial_i(F^2)\nu_j\right)} \, dy \\
& \qquad \ls \frac{1}{\delta} \int_{B_{\rho_2(s)}\backslash B_{\rho_1(s)}} \abs{F}\abs{\grad F} \frac{\delta^2 \wp(y)}{\wp(B(y))} \, dy := I_1^a \\
\end{align*}
And
\begin{align}\label{E:I1a_est}
\frac{1}{\delta} \int_0^{\delta} I_1^a \, ds &\ls \int_{B_{\beta_2\delta}
\backslash B_{\beta_1\delta}} \abs{F}\abs{\grad F} \frac{\delta
\wp(y)}{\wp(B(y))}
\, dy \\
& \ls
\widetilde{N}_{\tilde{\alpha}}(F)\widetilde{N}_{\tilde{\alpha}}(\delta\abs{\grad F})
\notag.
\end{align}
Recall that $\tilde{\alpha}$ must be chosen a little larger
than $\alpha$. The parameters $\beta_i$ determine the size of
$\tilde{\alpha}$, as we want all points in
$B_{\beta_2\delta}\subset \Gamma_{\tilde\alpha}(Q)$. The this
choice is irrelevant as long as $\Gamma_{\tilde\alpha}(Q)$ is
still a nontagential cone.

Next, we look at $\nu$, the outward unit normal for $B_r$ arising
from our first integration by parts. We know $\nu_j =
\frac{x_j}{\abs{x}}$ when $\abs{x} = r$, and $\partial_i(\nu_j) =
\frac{x_i x_j}{\abs{x}^3}$. Thus, for $x \in
B_{\rho_2(s)}\backslash B_{\rho_1(s)}$, $\beta_1 \delta \le r=\abs{x}
\le \beta_2 \delta$
 whence $\abs{\partial_i \nu_j} \ls \frac{1}{\delta}$. This leads to
\begin{align*}
&\frac{1}{\delta \omega_0(\Delta_{\delta}(Q))} \sum
\int_{B_{\rho_2(s)}\backslash
B_{\rho_1(s)}} \abs{G_0 a_0^{ij} F^2 \partial_i(\nu_j)} \, dy \\
& \qquad \qquad \qquad \ls \frac{1}{\delta^2}
\int_{B_{\rho_2(s)}\backslash B_{\rho_1(s)}} F^2 \frac{\delta^2
\wp(y)}{\wp(B(y))} \, dy := I_1^b,
\end{align*}
and
\begin{equation}\label{E:I1b_est}
\frac{1}{\delta} \int_0^{\delta} I_1^b \ls \int_{B_{\beta_2 \delta}
\backslash B_{\beta_1 \delta}} F^2 \frac{\wp(y)}{\wp(B(y))} \, dy \ls
\left(\widetilde{N}_{\tilde{\alpha}}(F)\right)^2.
\end{equation}

For the last term, we see
\begin{align}\label{E:I1c_est}
&\frac{1}{\delta^2 \omega_0(\Delta_{\delta}(Q))} \sum
\int_0^{\delta} \int_{\partial(B_{\rho_2(s)}\backslash
B_{\rho_1(s)})} \abs{G_0 a_0^{ij} F^2
\nu_i \nu_j} \, d\sigma \, ds \\
& \qquad \qquad \ls
\frac{1}{\delta^2\omega_0(\Delta_{\delta}(Q))} \int_{(B_{\beta_2
\delta} \backslash B_{\delta/5})\cup (B_{
\delta/6} \backslash B_{\beta_1\delta})} G_0 F^2 \, dy \notag\\
& \qquad \qquad \ls \int_{B_{\beta_2 \delta}} F^2
\frac{\wp(y)}{\wp(B(y))} \, dy \ls
\left(\widetilde{N}_{\tilde{\alpha}}(F)\right)^2.\notag
\end{align}

We now turn to estimating $I_2$; by the fact that $\omega$ is
doubling,
\begin{equation}\label{E:eps}
\abs{I_2} \ls \int_{B_r} \abs{F}\abs{\mathcal{L}_0 F}
\frac{G_0(0,y)}{\omega_0(\Delta_\delta(Q))} \, dy  \ls
\epsilon_0 \int_{B_r} \abs{F}\abs{\grad^2 u_1}
\frac{\delta^2(y)\wp(y)}{\wp(B(y))} \, dy.
\end{equation}
Here we are using the fact that ${\mathcal L}_0F=-{\mathcal
L}_0u_1=-({\mathcal
L}_1+\epsilon^{ij}\partial_{ij})u_1=-\epsilon^{ij}\partial_{ij}
u_1$ and $\sup_{z\in B_r}\abs{\epsilon^{ij}(z)}\le \sup_{z\in
B_{(x,\delta/2)}}\abs{\epsilon(z)}=\mathbf{a}(x)$. Using the condition
from Theorem \ref{T:perturbation}, we get that $\mathbf{a}(x)\ls
\epsilon_0$.

Thus,
\begin{align}\label{E:I2_est}
\frac{1}{\delta^2} \int_0^{\delta} \int_{\rho_1(s)}^{\rho_2(s)} \abs{I_2}
\, dr \, ds &\ls \frac{1}{\delta} \int_{\beta_1 \delta}^{\beta_2 \delta} \abs{I_2} \, dr\\
& \ls \epsilon_0 \int_{B_{\beta_2 \delta}} \abs{F}\abs{\grad^2 u_1}
\frac{\delta^2(y)\wp(y)}{\wp(B(y))} \, dy \ls \epsilon_0
\widetilde{N}_{\tilde{\alpha}}(F)A_{\tilde{\alpha}}(u_1) \notag
\end{align}

By combining \eqref{E:ntmax_dgradF_est}, \eqref{E:I1a_est},
\eqref{E:I1b_est}, \eqref{E:I1c_est} and \eqref{E:I2_est}, the
lemma is proven.
\end{proof}

\section{Proof of Lemma \ref{L:good-lambda}}

We now prove Lemma \ref{L:good-lambda}, the good-lambda inequality which is crucial for
estimating $S(F)$.

\begin{proof}[Proof of Lemma \ref{L:good-lambda}] Let us call $E$ the set
\begin{align*}
E&=\left\{Q\in\Delta; \, S_\beta(F)>2\lambda,\,
\widetilde{N}_\alpha(F)\le\gamma\lambda, \right. \\
 & \qquad \qquad \quad \left. \widetilde{N}_\alpha(\delta\abs{\grad
F})\le\gamma\lambda, \widetilde{N}_\alpha(F)A_\alpha(u_1)\le
(\gamma\lambda)^2 \right\}.
\end{align*}
It is sufficient to prove that $\omega(E)\le
C\gamma^2\omega(\Delta)$, since we already know that $\omega\in
A_\infty(d\sigma)$. Standard arguments (see \cite{DJK} or
\cite{FKP}) show that since $S_\beta(F)>2\lambda$ on $E$, we
can choose $\gamma>0$ sufficiently small so that $S_{\beta,\tau
r}(F)>\lambda/2$, where each cone is truncated at height $\tau
r$ for some fixed $0<\tau<1$. It follows that
\begin{equation}
\omega(E)\le \frac4{\lambda^2}\int_E S^2_{\beta,\tau r}(F)(Q)
d\omega(Q).\label{eq1}
\end{equation}

We will introduce the following notation. Let
\begin{equation}
\mathcal{ D}:=\bigcup_{Q\in E}\Gamma_{\beta,\tau r}(Q).
\end{equation}
For $\alpha'\in (\beta,\alpha)$, let us consider a smoothed-out
version of the set $\bigcup_{Q\in E}\Gamma_{\alpha',\tau
r}(Q)$. We denote by $\mathcal{ D}_{\alpha'}$ the set with the
properties:
\begin{enumerate}[(i)]
\item $\mathcal{ D}\subset \bigcup_{Q\in E}\Gamma_{\alpha',\tau
r}(Q)\subset \mathcal{ D}_{\alpha'}\subset \bigcup_{Q\in
E}\Gamma_{\alpha,3\tau
r}(Q)$,
\item $\partial\mathcal{ D}_{\alpha'}$\mbox{ is smooth except at $E$ and
$\abs{\grad \nu(Q)}\le C/\delta_{D}(Q)$ for $Q\in \partial\mathcal{
D}_{\alpha'}$},
\item $\mathcal{D}_{\alpha'}\subset \mathcal{
D}_{\alpha''}\quad\mbox{if }\alpha'<\alpha''.$
\end{enumerate}
Here $\nu$ is the outer normal at the boundary and
$\delta=\delta_{D}$ denotes the distance to the boundary of the
original domain $D$. We now work with \eqref{eq1}.
\begin{eqnarray}
\omega(E)&\le&
\frac4{\lambda^2}\int_{E}\left(\int_{\Gamma_{\beta,\tau
r}(Q)}\abs{\grad F}^2\frac{{\delta^2\wp}(x)}{{\wp}(B(x))}dx
\right)\,d\omega(Q)\nonumber\\
&\le&\frac{C}{\lambda^2}\int_{E}\left(\int_{\Gamma_{\beta,\tau
r}(Q)}\abs{\grad F}^2G_0\frac{dx}{\omega(\Delta_x)}
\right)\,d\omega(Q)\label{eq2}\\
&\le&\frac{C}{\lambda^2}\int_\mathcal{ D_{\alpha'}}\abs{\grad
F}^2G_0\,dx\nonumber\\
&\le&\frac{C}{\lambda^2}\int_\mathcal{ D_{\alpha'}} (A_0\grad
F\cdot\grad F)G_0\,dx.\nonumber
\end{eqnarray}
Here $\Delta_x=\{Q\in \partial D; x\in\Gamma_\beta(Q)\}$ and
$\alpha'\in (\beta,\alpha)$. Now, \[A_0\grad F\cdot\grad
F=\mathcal{ L}_0(F^2)-2F\mathcal{ L}_0F,\] so there are two terms
to estimate
\begin{eqnarray}
&&\int_\mathcal{ D_{\alpha'}} (A_0\grad F\cdot\grad
F)G_0\,dx\nonumber\\&=&\int_\mathcal{ D_{\alpha'}}\mathcal{
L}_0(F^2)G_0\,dx-\int_\mathcal{ D_{\alpha'}}2F\mathcal{
L}_0F\,G_0\,dx.\label{eq3}
\end{eqnarray}
Let us denote these two terms by $I_1$ and $I_2$. We first deal
with $I_1$. Recall that $\mathcal{ L}_0G_0=-\delta(0)$, hence
integration by parts gives us only two boundary terms
\begin{eqnarray}
I_1&\le&\int_{\partial\mathcal{
D}_{\alpha'}}a_0^{ij}\partial_i(F^2)G_0\nu_j\,d\sigma-\int_{\partial\mathcal{
D}_{\alpha'}}
\partial_j(a_0^{ij}G_0)F^2\nu_i\,d\sigma.\label{eq4}
\end{eqnarray}
Note that, strictly speaking, these two boundary terms are not well-defined.
 To fix this, we again use the averaging technique
introduced before. We integrate over the interval
$[\alpha',\alpha'']\subset (\beta,\alpha)$ and get instead solid
integrals
\begin{eqnarray}
I_1&\le&c\abs{\int_{\mathcal{D}_{\alpha''} \setminus \mathcal{D}_{\alpha'}}a_0^{ij} \partial_i(F^2)G_0\delta^{-1}\nu_j\,dx} \nonumber \\
 & & \qquad + c\abs{\int_{\mathcal{D}_{\alpha''}\setminus\mathcal{D}_{\alpha'}}\partial_j(a_0^{ij}G_0)F^2\delta^{-1}\nu_i\,dx}\nonumber\\
&=&I_1^a+I_1^b.\label{eq5}
\end{eqnarray}
Now, for simplicity let
$\widetilde{\mathcal{D}}=\mathcal{D}_{\alpha''}\setminus\mathcal{D}_{\alpha'}$.
We see that

\begin{eqnarray}
I^a_1&\le&C\int_{\widetilde{\mathcal{D}}}\abs{F}\abs{\grad
F}G_0\delta^{-1}dx \nonumber\\
&\eqs& \int_{Q\in2\Delta}\left(\int_{\Gamma_{\beta}(Q)\cap
\widetilde{\mathcal{D}}}\abs{F}\abs{\grad
F}G_0\frac{dx}{\delta\omega(\Delta_x)}\right)d\omega(Q)\nonumber\\
&\eqs& \int_{Q\in2\Delta}\left(\int_{\Gamma_{\beta}(Q)\cap
\widetilde{\mathcal{D}}}\abs{F}\abs{\grad
F}\delta\frac{\wp(x)}{\wp(B(x))}dx\right)d\omega(Q).
\nonumber\\
&\le& \int_{Q\in2\Delta}\left(\int_{\Gamma_{\beta}(Q)\cap
\widetilde{\mathcal{D}}}\abs{F}\frac{\wp(x)}{\wp(B(x))}dx\right)^{1/2}\times\label{eq6}\\&&\qquad\left(\int_{\Gamma_{\beta}(Q)\cap
\widetilde{\mathcal{D}}}\abs{\grad
F}^2\delta^2\frac{\wp(x)}{\wp(B(x))}dx\right)^{1/2}d\omega(Q).\nonumber
\end{eqnarray}
The key is that if $x\in\Gamma_{\beta}(Q)\cap
\widetilde{\mathcal{D}}$ then $x\in\Gamma_{\alpha}(Q')$ for some
$Q'\in E$ and the set $\Gamma_{\beta}(Q)\cap
\widetilde{\mathcal{D}}$ is of diameter proportional to
$\delta(x)$ and its distance to $\partial D$ is also of
$\delta(x)$ size. This implies we can control the two solid
integrals the the last line by
$\widetilde{N}_\alpha(F)(Q')\widetilde{N}_\alpha(\delta\abs{\grad
F})(Q')$. This gives
\[I_1^a\le C\int_{Q\in
2\Delta}(\gamma\lambda)^2d\omega(Q)=C\gamma^2\lambda^2\omega(\Delta),\]
since the measure $\omega$ is doubling. To estimate $I_1^b$ we
integrate by parts one more time. We get
\begin{eqnarray}
I_1^b&\le&c\abs{\int_{\widetilde{\mathcal{D}}}
a_0^{ij}G_0\partial_j\left(\frac{F^2\nu_i}{\delta}\right)dx}+c\abs{\int_{\partial
\widetilde{\mathcal{D}}}
a_0^{ij}G_0{F^2\nu_i\nu_j}{\delta}^{-1}d\sigma}.\label{eq7}
\end{eqnarray}
The first term of \eqref{eq7} will give us two additional terms,
depending on where the derivative $\partial_j$ falls. By the chain
rule,
\[\abs{\partial_i\left(\frac{F^2\nu_i}{\delta}\right)}\le
\frac{C\abs{F}\abs{\grad F}}{\delta}+\frac{CF^2}{\delta^2}.\] Here we use
the fact that the real distance function $\delta$ can be replaced
by a smooth distance function so that $\abs{\grad \delta^{-1}}\eqs
\delta^{-2}$ and also $\abs{\grad \nu_i}\le C\delta^{-1}$. Hence, the
first term is of the same type as $I_1^a$, and the second one can
be bounded by
\begin{equation}
c\int_{\widetilde{\mathcal{D}}} F^2{G_0}{\delta^{-2}}\,dx\eqs
\int_{Q\in 2\Delta}\left(\iint_{\Gamma_{\beta}(Q)\cap
\widetilde{\mathcal{D}}}\abs{F}^2\delta\frac{\wp(x)}{\wp(B(x))}dx\right)d\omega(Q).\label{eq8}
\end{equation}
Thus, this term can be dominated by $C\int_{Q\in 2\Delta}
\left(\widetilde{N}_\alpha(F)(Q')\right)^2d\sigma\le
C\gamma^2\lambda^2\omega(\Delta).$ Finally, \eqref{eq7} has one
additional boundary term, which again has to be averaged out. So
we need to use the wiggling technique one more time. Without going
into too much detail, this will again turn the surface integral
into a solid integral over a set we call
$\widetilde{\mathcal{D'}}$ (essentially of the same type as
$\widetilde{\mathcal{D}}$):
\[\abs{\int_{\widetilde{\mathcal{D'}}}
a_0^{ij}G_0\frac{F^2\nu_i\nu_j}{\delta^2}dx}\ls
\int_{\widetilde{\mathcal{D'}}} F^2{G_0}{\delta^{-2}}\,dx.\]
Notice that this term is similar to \eqref{eq8}, so the same
estimates can be applied. This establishes
\[\abs{I_1^a}\le
C\gamma^2\lambda^2\omega(\Delta).\] Now we deal with $I_2$. As
before, we use $\mathcal{ L}_0 F=-\epsilon^{ij}\partial_{ij} u_1$,
where $\epsilon^{ij}=a_0^{ij}-a_1^{ij}$. This gives
\begin{eqnarray}
I_2&\le&C\int_\mathcal{ D_{\alpha'}}\epsilon(x)\abs{F}\abs{\grad^2
u_1}\,G_0\,dx,\label{eq9}
\end{eqnarray}
where $\epsilon(x)=\max \abs{\epsilon_{ij}(x)}$. We turn this
back (by Fubini) to into two integrals
\begin{eqnarray}
I_2&\ls&\int_{Q\in 2\Delta} \left(\int_{\Gamma_{\alpha'}(Q)\cap
\mathcal{ D}_{\alpha'}}\epsilon(x) \abs{F}\abs{\grad^2
u_1}\,\frac{G_0}{\omega(\Delta_x)}\,dx\right)d\sigma(Q)\label{eq10}\\
&\eqs&\int_{Q\in 2\Delta} \left(\int_{\Gamma_{\alpha'}(Q)\cap
\mathcal{ D}_{\alpha'}}\epsilon(x)\abs{F}\abs{\grad^2
u_1}\delta^2\frac{\wp(x)}{\wp(B(x))}\,dx\right)d\sigma(Q).\nonumber
\end{eqnarray}
By H\"older:
\begin{eqnarray}
I_2&\ls&\int_{Q\in 2\Delta} \left(\int_{\Gamma_{\alpha'}(Q)\cap
\mathcal{
D}_{\alpha'}}\epsilon(x)^2\abs{F}^2\frac{\wp(x)}{\wp(B(x))}\,dx\right)^{1/2}\times\label{eq11}\\
&&\quad\left(\int_{\Gamma_{\alpha'}(Q)\cap \mathcal{
D}_{\alpha'}}\abs{\grad^2
u_1}^2\delta^4\frac{\wp(x)}{\wp(B(x))}\,dx\right)^{1/2}d\sigma(Q).\nonumber
\end{eqnarray}
As $\alpha'<\alpha$, it can be arranged that either
$\Gamma_{\alpha'}(Q)\cap \mathcal{ D}_{\alpha'}\subset
\Gamma_{\alpha}(Q')$ for some $Q'\in E$ or
$\widetilde{N}_\alpha(F)(Q)A_\alpha(u_1)(Q)\le (\gamma\lambda)^2$.

Indeed, if $Q\in \overline{E}$, then the fact that
$\widetilde{N}_\alpha(F)(Q_n)A_\alpha(u_1)(Q_n) \le
(\gamma\lambda)^2$ for a sequence of $Q_n\in E$ converging to $Q$
implies the same for $Q$. In this case we just take $Q'=Q$.

Otherwise $d=\dist{Q,E}>0$, and this gives that
$\Gamma_{\alpha'}(Q)\cap \mathcal{ D}_{\alpha'}$ only contains
points of distance $\delta \gs d$. Hence by making $\alpha$
sufficiently large we will have $\Gamma_{\alpha'}(Q)\cap \mathcal{
D}_{\alpha'}\subset \Gamma_{\alpha}(Q')$ for all points $Q'\in E$
such that $\dist{Q,Q'}\eqs d$. If follows that

\[\left(\int_{\Gamma_{\alpha'}(Q)\cap \mathcal{
D}_{\alpha'}}\abs{\grad^2
u_1}^2\delta^4\frac{\wp(x)}{\wp(B(x))}\,dx\right)^{1/2}\le
A_{\alpha}(u_1)(Q').\] On the other hand,
\begin{equation}
\left(\iint_{\Gamma_{\alpha'}(Q)\cap \mathcal{
D}_{\alpha'}}\epsilon(x)^2\abs{F}^2\frac{\wp(x)}{\wp(B(x))}\,dx\right)^{1/2}\le
E_{3\tau r}(Q)\widetilde{N}_{\alpha}(F)(Q'),\label{eq12}
\end{equation}
where
\[E_{3\tau r}(Q)=\left(\int_{\Gamma_{\alpha,3\tau r}(Q)}
\frac{\left(\sup_{B(x,\,\delta(x)/6)}\epsilon(x)\right)^2}{\delta^n}\,dx\right)^{1/2}.\]
To see \eqref{eq12} we cover the set $\Gamma_{\alpha'}(Q)\cap
\mathcal{D}_{\alpha'}$ by a union of balls of diadic diameters
$2^k r$, $k\in \Z$, with each such ball of approximate distance $2^k r$
to the boundary such that each point $x\in \Gamma_{\alpha'}(Q)\cap
\mathcal{D}_{\alpha'}$ belongs to at most $K$ balls. (Simple
geometric considerations imply that $K$ will only depend on the
dimension, the number $\alpha'$ and the Lipschitz constant of $D$). On
each such ball, the square of the solid integral on the left-hand
side of \eqref{eq12} can be estimated by $C\left(\sup_{x\in
B^i}\epsilon(x)^2\right)\widetilde{N}_\alpha(F)^2(Q')$. After
we sum over all the balls we get the expression $CKE^2_{3\tau
r}(Q)\widetilde{N}_{\alpha}(F)^2(Q')$. It follows that
\begin{eqnarray}
I_2&\le&C\int_{Q\in
2\Delta}A_{\alpha}(u_1)(Q')\widetilde{N}_{\alpha}(F)(Q')E_{3\tau
r}(Q)\,d\omega(Q)\nonumber\\
&\le& C\gamma^2\lambda^2\int_{Q\in 2\Delta}E_{3\tau
r}(Q)\,d\omega(Q)\\
&\le& C\gamma^2\lambda^2\omega(2\Delta)^{1/2}\left(\int_{Q\in
2\Delta}E^2_{3\tau
r}(Q)\,d\omega(Q)\right)^{1/2}\nonumber\\
&\le& C\gamma^2\lambda^2\omega(2\Delta)\le
C'\gamma^2\lambda^2\omega(\Delta).\nonumber
\end{eqnarray}
since
\[\int_{Q\in 2\Delta}E^2_{3\tau r}(Q)\,d\omega(Q)\le
C\omega(2\Delta)\] by Rios's work (see p. 683 of \cite{R1}). Note
that this is the only place we are using \eqref{E:perturbation},
and we do not use the fact that it is small, only that it is finite. This establishes the
good-lambda lemma.\qed

\begin{corollary}\label{c1} Lemma \ref{L:good-lambda} implies that for any
$1<p<\infty$:
\begin{equation}
\int_{\partial D} S(F)^pd\,\sigma\le C(q) \int_{\partial D}
\left((\widetilde{N}(F)^p+\widetilde{N}(\delta \abs{\grad
F})^p\right) d\sigma + \int_{\partial D} S(u_0)^p\,d\sigma,
\label{eq13}
\end{equation}
where the square function $S$ is defined over cones of smaller
aperture than the modified nontangential maximal function
$\widetilde{N}$.
\end{corollary}

\noindent {\it Proof.} Indeed, the Whitney decomposition and
Lemma \ref{L:good-lambda} gives us
\begin{align*}
\int_{\partial D} S_\beta(F)^pd\,\sigma &\le C\left[ \int_{\partial
D} \left(\widetilde{N}_\alpha(F)^p + \widetilde{N}_\alpha(\delta
\abs{\grad F})^p\right) d\sigma \right. \\
& \qquad \qquad + \left. \int_{\partial D}
(A_\alpha(u_1)\widetilde{N_\alpha}(F))^{p/2}\,d\sigma\right]
\end{align*}
for some $\beta<\alpha$. This implies that for any
$\epsilon>0$,
\begin{equation*}
\int_{\partial D} S_\beta(F)^pd\,\sigma\le C(\epsilon)
\int_{\partial
D}\left(\widetilde{N}_\alpha(F)^p+\widetilde{N}_\alpha(\delta
\abs{\grad F})^p\right) d\sigma + \epsilon\int_{\partial
D}A_\alpha(u_1)^p\,d\sigma.
\end{equation*}
By Theorem 2.19 of \cite{R1}, since $u_1$ is a solution to
$\mathcal{ L}_1u_1=0$ we have pointwise estimate $A_\alpha(u_1)\le
C S_{c\alpha}(u_1)$ for some $c>1$ depending only on the dimension
$n$. Also by \cite{EK} (see also Theorem 2.17 of \cite {R1}) for
solutions we have $\norm{S_{c\alpha}(u_1)}_{L^p}\le
C\norm{S_{\beta}(u_1)}_{L^p}$ with $C$ only depending on the
ellipticity constant, the numbers $c\alpha$ and $\beta$ and the
dimension.

This gives
\begin{equation*}
\int_{\partial D} S_\beta(F)^pd\,\sigma\le C(\epsilon)
\int_{\partial D}
\left(\widetilde{N}_\alpha(F)^p+\widetilde{N}_\alpha(\delta
\abs{\grad F})^p\right) d\sigma + C_1\epsilon \int_{\partial
D}S_{\beta}(u_1)^p\,d\sigma.
\end{equation*}

We can write $S_{\beta}(u_1)^p\le
C_2(S_{\beta}(u_0)^p+S_{\beta}(F)^p)$. Choose $\epsilon$ so
that $C_1C_2\epsilon<1/2$ (this allows the term
$C_1C_2\epsilon S_{\beta}(F)^p$ to be incorporated into the right-hand side). It
follows that
\begin{equation*}
\int_{\partial D} S_\beta(F)^pd\,\sigma\le 2C(\epsilon)
\int_{\partial D}
\left(\widetilde{N}_\alpha(F)^p+\widetilde{N}_\alpha(\delta
\abs{\grad F})^p\right) d\sigma + \int_{\partial
D}S_{\beta}(u_0)^p\,d\sigma.
\end{equation*}
\end{proof}

\end{document}